\documentclass[12points,reqno]{amsart}
\usepackage{amsfonts}
\usepackage{amssymb}
\usepackage{graphicx}
\usepackage{amsmath}

\newtheorem{theorem}{Theorem}
\newtheorem{lemma}{Lemma}

\theoremstyle{definition}

\theoremstyle{remark}

\numberwithin{equation}{section}

\everymath{\displaystyle}

\newcommand*\un{\operatornamewithlimits{\cup}}

\begin{document}
\title{ Embeddings in Lie algebras of subexponential growth }

%

\author{Adel Alahmadi}
\address{Department of Mathematics, Faculty of Science, King Abdulaziz University, P.O.Box 80203, Jeddah, 21589, Saudi Arabia}
\email{analahmadi@kau.edu.sa}
\author{Hamed Alsulami}
\address{Department of Mathematics, Faculty of Science, King Abdulaziz University, P.O.Box 80203, Jeddah, 21589, Saudi Arabia}
\email{hhaalsalmi@kau.edu.sa}
%
%

\maketitle

\begin{abstract}
We prove that an arbitrary countable dimensional Lie algebra over a field of characteristic $\neq 2$ that is locally of subexponential growth is embeddable in a finitely generated Lie algebra of subexponential growth.
\end{abstract}

\maketitle

\section{Introduction and  Main Result}

 Throughout the paper all algebras are considered over a field $F$ of characteristic $\neq 2.$ Let $A$ be a (not necessarily associative) algebra generated by a finite dimensional subspace $V$. Let $V^{n}, n\geq 1,$ be the subspace of the algebra $A$ spanned by all products $v_1\cdots v_k, k \leq n, v_i \in V,$ with all possible arrangements of brackets. We have $V = V^1 \leq V^2\cdots ; \un\limits_n V^n=A; \,\,dim_F V^n < \infty, n \geq 1.$ The function $g(V,n)=dim_F V^n$ is called the growth function of $A$ with respect to the generating subspace $V.$

Let $N=\{1 , 2, \cdots\}$ denote the set of positive integers. Given two increasing functions $f,g: N \rightarrow [1, \infty),$  we say that $f$ is asymptotically less or equal to $g\, (f \preceq g)$  if there exists $C \in N$ such that $f(n) \leq C g(Cn)$ for all $n \geq 1.$
If $f \preceq g$ and $g \preceq f$ then we say that $f$ and $g$ are asymptotically equivalent, $f \sim g.$ For any two finite dimensional generating subspaces $V, W$ of the algebra $A$ we have $g (V, n) \sim g(W, n).$ By the growth $g_A$ of the algebra $A$ we mean the class of equivalence of $g(V, n).$

A function $f: N \rightarrow [1, \infty)$ is said to be subexponential if $\displaystyle\lim\limits_{n\to \infty}\frac{f(n)}{e^{\alpha n}}=0$ for any $\alpha > 0.$ An algebra $A$ is of subexponential growth if $g(V, n)$ is a subexponential function. For a not necessarily finitely generated  algebra $A$, we say that $A$  is locally of subexponential growth if every finitely generated subalgebra of $A$ is of subexponential growth.

In 1949, G. Higman, B. Neumann and H. Neumann [4] proved that an arbitrary countable group is embeddable in a finitely generated group. A. I. Malcev [5] and A. I. Shirshov [6] proved analogous results for countable dimensional associative and Lie algebras respectively. In [3] L. Bartholdi and A. Erschler showed that an arbitrary countable group  that is locally of subexponential growth is embeddable in a finitely generated group of subexponential growth. The key role in their construction was played by wreath products with the infinite cyclic group. In [2], we used matrix wreath products of algebras to embed an arbitrary countable dimensional associative algebra that is locally of subexponential growth into a finitely generated associative algebra of subexponential growth. In this paper, we obtain a similar result for Lie algebras.

\medskip

\begin{theorem}
 Let char $F \neq 2.$ An arbitrary countable dimensional Lie $F-$algebra that is locally of subexponential growth is embeddable in a finitely generated Lie algebra of subexponential growth.
 \end{theorem}

 \begin{lemma}
 Let $A$ be a countable dimensional associative algebra that is locally of subexponential growth. Then there is an embedding $\varphi:A\rightarrow B$ into another countable dimensional associative algebra that is locally of subexponential growth, such that $\varphi(A)\subseteq [B, B].$
\end{lemma}

 \begin{proof}
 Without  loss of generality, we will assume that the algebra $A$ contains an identity element $1.$
 Let $N=\{1,2,\cdots \}$ be the set of positive integers. Consider the algebra $M_{N \times N}(A)$ of infinite $N \times N$ matrices over $A$ having finitely many nonzero entries in each column.
 For an element $a \in A$ and integers $i, j \in N,\,$ let $e_{i,j}(a)$ denote the matrix having the element $a$ at the intersection of the $i-$th row and $j-$th column and zeros everywhere else.
 Consider the subalgebra $B$ of $M_{N \times N}(A)$ generated by infinite matrices $E_1(a)=\displaystyle\sum_{i=1}^{\infty}e_{i,i+1}(a), a \in A,$ and $E_{-1}(1)=\displaystyle\sum_{i=1}^{\infty}e_{i+1,i}(1)$
 We have $[E_{1}(a), E_{-1}(1)]=e_{1,1}(a).$
 Hence, $\varphi:a\rightarrow e_{1,1}(a), a\in A,$ is an embedding of the algebra $A$ into the algebra $B$ and $\varphi(A) \subseteq [B, B].$
 It remains to show that the algebra $B$ is locally of subexponential growth. Choose arbitrary elements $a_1,...,a_m\in A.$ We need to prove that the algebra $B^{'}$ generated by $E_1 (a_1),..., E_{1}(a_m),E_{-1}(1)$ has subexponential growth.

 Consider the finite dimensional subspaces $V=Span_{F} (a_1,...,a_m, 1) \subset A,$ $W=Span_{F} (E_1(a),...,E_1(a_m), E_{-1}(1))\subset B.$
 For $k \geq 0$ and an element $a \in A$ denote $E_k(a)=\displaystyle\sum_{i=1}^{\infty}e_{i,i+k}(a).$ For $k \leq -1$ and an element $a \in A$ denote $E_k(a)=\displaystyle\sum_{i=1}^{\infty}e_{i-k,i}(a).$ We will show that for an arbitrary $n \geq 1$

 \begin{align}
& W^{n} \subseteq M_{[1, n]\times [1,n]} (V^{n})+\displaystyle\sum_{i=-n}^{n}E_{i}(V^{n}), \label{eq1}
\end{align}

where $M_{[1, n]\times [1,n]} (V^{n})=\displaystyle\sum_{1 \leq i,j \leq n}F e_{i,j}(V^{n}).$
Choose two nonnegative integers $p, q \in \mathbb{Z};\, p, q \geq 0,$ and elements $a, b \in A.$
Straightforward computations show that $E_p (a) E_q (b)=E_{p+q} (ab),\,\, E_{-p} (a) E_{-q} (b)= E_{-p-q}(ab),\,\, E_p (a) E_{-q} (b)=E_{p-q}(ab).$ However, $E_{-q}(a) E_{p}(b)=\displaystyle\sum_{i=1}^{\infty}e_{i+q,i}(a).\displaystyle\sum_{j=1}^{\infty}e_{j,j+p}(b)$

\[= \begin{cases}
      E_{p-q} (ab)-\displaystyle\sum_{i=1}^{q}e_{i,i+p-q}(ab) & if~ p \geq q, \\
      E_{p-q} (ab)-\displaystyle\sum_{i=1}^{p}e_{i+q-p,i}(ab) & if~ p < q.
   \end{cases}
\]
 Finally, $e_{i,j} (a) E_{\pm p} (b)=e_{i,j\pm p} (ab)$ if $j \pm p \geq 1.$ If $j \pm p < 1$ then $e_{i,j}(a) E_{\pm p}(b)=0.$ Similarly, $E_{\pm p}(a) e_{i,j}(b)=e_{\mp p+i,j} (ab)$, if $i \mp p \geq 1$ and $0$ otherwise. The multiplication rules above imply the inclusion (\ref{eq1}). From the inclusion (\ref{eq1}) it follows that $g_{B^{'}}(W,n)\leq n^{2} g_{A} (V, n)+(2n+1) g_{A} (V,n).$ Hence, the function $g_{B^{'}}(W,n)$ is subexponential. This completes the proof of the lemma.
 \end{proof}

\begin{proof}[Proof of Theorem 1.]
Let $L$ be a countable dimensional Lie algebra that is locally of subexponential growth. Let $U=U(L)$ be the universal associative enveloping algebra of $L.$ As shown by M. Smith [7] the algebra $U$ is locally of subexponential growth. By the lemma, the algebra $U$ is embeddable in a countable dimensional algebra $A$, that is locally of subexponential growth and the image of $U$ is contained in the subspace $[A, A].$
By the Theorem 2  from [2], the algebra $A$ is embeddable in a finitely generated algebra $B$ of subexponential growth. Without loss of generality, we will assume that $B \ni 1.$

Consider the algebra $C=M_2(B)$ of $2\times2$ matrices over $B.$ In [1], we showed that if a finitely generated algebra $R$ over a field of characteristic $\neq 2$ contains  an idempotent $e$ such that $R e R=R(1-e) R=R$ then the Lie algebra $[R, R]$ is finitely generated. This implies that the Lie algebra $[C, C]$ is finitely generated. The growth of the Lie algebra $[C, C]$ is less or equal to the growth of the associative algebra $C$ that is asymptotically equivalent to the growth of the algebra $B.$ Hence, $[C, C]$ is a finitely generated Lie algebra of subexponential growth.

We have the embeddings $L \rightarrow U^{(-)}\rightarrow [A, A]\rightarrow [B, B]\rightarrow [C, C],$ which completes the proof of the theorem.
\end{proof}

\end{document}